\documentclass[10pt]{article}
\usepackage{amssymb,amsfonts}
\usepackage{pb-diagram}
\usepackage{amsmath,amsthm,amsfonts,amssymb}
\usepackage[usenames]{color}
\usepackage{hyperref}
\usepackage{soul}
\usepackage{graphicx}
\usepackage{amssymb}
\usepackage{amsmath}
\usepackage{amssymb}

\newtheorem{theorem}{Theorem}
\newtheorem{definition}{Definition}

\newtheorem{lemma}{Lemma}
\newtheorem{cor}{Corollary}

\newtheorem{conjecture}{Conjecture}

\newcommand{\be}{\begin{enumerate}}
\newcommand{\ee}{\end{enumerate}}
\newcommand{\beq}{\begin{equation}}
\newcommand{\eeq}{\end{equation}}

\def\N{{\mathbb{N}}}
\def\Z{{\mathbb{Z}}}

\def\Q{{\mathbb{Q}}}

\def\MA{{\mathbb{A}}}

\def\B{{\mathcal{B}}}
\def\A{{\mathcal{A}}}
\def\M{{\mathcal{M}}}

\def\C{{\mathcal{C}}}

{}
{}

\title{Equations in Algebras}
\author{Olga Kharlampovich \footnote{Hunter College, CUNY, Supported by NSF grant DMS-1201379 and PSC-CUNY  award} and Alexei Myasnikov \footnote{Stevens Institute of Technology, supported by NSF grant DMS-1201379}}

\date{}

\begin{document}

\maketitle

\begin{abstract}   We  show that   the Diophantine problem(decidability of equations)  is undecidable in free associative algebras over any field and in  the group algebras over any field of  a wide variety  of  torsion free groups, including toral relatively hyperbolic groups, right angled Artin groups, commutative transitive groups, the  fundamental groups of various graph groups,  etc.

\end{abstract}

\section{Introduction}

In this paper we study Diophantine problems (decidability of equations) in free associative algebras over an arbitrary field and group algebras for  a wide class of groups. We show that some kind of  arithmetic can be always interpreted by equations in  all these algebras, hence undecidability of Diophantine problems. Our main approach is to reduce the Diophantine problem in each of these algebras to similar  problems in various  polynomial rings.

Study of equations in algebra has a rich and long history. The famous Hilbert tenth problem stated in 1900 asked for  a procedure which, in a finite number of steps, can determine whether a polynomial equation (in several variables) with integer coefficients has or does not have integer solutions.
In 1970  Matiyasevich, following the work of Davis, Putnam and Robinson, solved this problem in the  negative  \cite{mat}.
Similar  questions can be asked for arbitrary commutative rings $R$.  The {\em Diophantine problem for a given commutative ring $R$} asks if there exists  an algorithm that decides whether or not a given polynomial equation (a finite system of polynomial equations) with coefficients in some subring $R_0$ of  $R$  has a solution in $R$. In this case elements of $R_0$ must be recognizable by computers, so we always assume that $R_0$ is a computable ring. Now we mention  a few principal  results on Diophantine problems in rings. For the following fields the  Diophantine problem  is decidable, in fact, the whole their  first-order theory is decidable: the field of complex numbers
$\mathbb{C}$ where   coefficients are taken in the algebraic closure of the  rationals $\bar{\mathbb{Q}}$, 
the field of real numbers $\mathbb{R}$, where coefficients are computable reals,  and  $p$-adic numbers
$\mathbb{Q}_p$ \cite{ax3, ershov}, where coefficients are computable $p$-adics. On the other hand, 
 undecidability of the Diophantine problem was proved for rings of polynomials $R[X]$ over an integral domain $R$ \cite{denef,denef2},
and for  rings of Laurent polynomials $R[X,X^{-1}],$  \cite{pappas,pheidas}.  

A major open problem  is the Diophantine problem (sometimes called  generalized  Hilbert's tenth problem) for the field ${\mathbb Q}$ of rational
numbers (see  a comment on this  in the next section). A  survey of the results on the undecidability of existential theories of rings and fields is given in \cite{pheidas_zahidi}.

Let $\MA_K(A)$ be a free associative algebra with basis $A$ over a field $K$.   
An equation with variables in $X = \{x_1, \ldots,x_n\}$ and constants from $\MA_K(A)$ is an expression 
$P(X,A) = 0$
 where $P(X,A)$ is an element from $\MA_K(A\cup X)$.  
 A solution to an equation $P(X,A) = 0$ is a map $\phi:X  \to u_i \in \MA_K(A)$ such that $P(X^\phi,A) = 0$ in $\MA_K(A)$.
 Of course the coefficients of the equation above have to be given by  an effective description, so one may assume that they are taken from a constructive (computable) subring of  $\MA_K(A)$, typically of the form $\MA_{K_0}(A)$, where $K_0$ is a computable subfield of $K$, say the prime subfield of $K$. Of course,  in the case of  undecidable Diophantine problem, the smaller the subfield $K_0$ is the stronger the result. We prove that for any non-empty set $A$ and for any field $K$ the Diophantine problem for $\MA_K(A)$ with coefficients in $\MA_{K_0}(A)$ where $K_0$ is the prime subfield of $K$ is undecidable. To approach this result we  use   Pell's equations of a particular type in $\MA_K(A)$ to interpret arithmetic by equations in $\MA_K(A)$. 
 Observe, that Makanin proved in \cite{Makanin1} that the Diophantine problem in a free monoid is decidable, so in the case when the field $K$ has decidable Diophantine problem (say when $K$ is finite or algebraically closed) undecidability of the Diophantine problem in the algebra $\MA_K(A)$ cannot be seen on the level of coefficients or monomials. We showed in Section \ref{se:bounded-width} that decidability of the Diophantine problem in $K$ is equivalent to solvability of equations when solutions are assumed to have  bounded width (the number of monomials with non-zero coefficients).

One can consider equations and their decidability over arbitrary algebraic structures  $\mathcal{M}$  in a language $L$. 
An equation in  $\mathcal{M}$ is an equality of two terms in $L$:
$$
t(x_1, \ldots, x_n, a_1, \ldots, a_m) = s(x_1, \ldots, x_n, b_1, \ldots, b_k).
$$
with variables $x_1, \ldots,x_n$ and constants $a_1, \ldots, a_m, b_1, \ldots, b_k  \in \mathcal{M}$.  A solution of such an equation is a map  $x_i \to c_i$ from the set of variables into $\mathcal{M}$ which turns the symbolic equality of terms $t = s$ into a true   equality in $\mathcal{M}$. In particular,  one can consider equations in semigroups, groups,  associative or Lie algebras, etc.

The following are the principal    questions on equations in  $\mathcal{M}$:
decidability of single equations and finite systems (Diophantine problems),
equivalence of  infinite systems of equations  in finitely many variables  to some of their  finite subsystems (equationally Noetherian structures),
description of solution sets of finite systems of equations. In  generalized Hilbert's 10th problems the question is usually stated about  decidability of  single equations, since in those cases finite systems of equations are equivalent (have the same solution sets) as suitable single equations. However, in general, this may not be the case, so in the decidability results the statements are stronger for finite systems, while in the undecidable cases,  formulations are stronger for single equations.

The principal questions are solved positively in 
abelian groups (linear algebra), 
 free groups \cite{Makanin},\cite{Razborov}, \cite{irc},
hyperbolic  and toral relatively hyperbolic groups \cite{Rips-Sela},\cite {Dahmani-Groves},
Right angled Artin groups \cite{diek}, \cite{CRK} and free products of groups  with decidable  equations in the factors  (see also  \cite{CRK1}), and some other groups. 

In the second part of the paper we study Diophantine problems in group algebras of torsion-free groups.
Let $K(G)$ be  the  group algebra $K(G)$ of a group $G$ over a field $K$.  We  show that   the Diophantine problem is undecidable in $K(G)$ for any  field $K$ for a wide variety of groups $G$.

The main technical result is Theorem \ref{main} which  states  that if a torsion free group $G$  contains an element $g$ such that 
the centralizer $C_G(g)$ is a countable free abelian  group and 
 $C_G(g^k) = C_G(g)$ for any non-zero integer $k$, then  the Diophantine problem in $K(G)$ is undecidable.  
This covers large variety of torsion-free groups $G.$

 Note that for any group $G$ as above one can effectively construct a particular system of equations with two parameters $\nu,  x \in K(G)$ such that one cannot recognize for which parameters $\nu$ the system has a solution in $K(G)$.

 The results above show that decidability of the Diophantine problem in $K(G)$ does not directly depend  on the Diophantine problems in $K$ and $G$. To clarify the situation in the last section we introduce the notion of the Bounded Width  Diophantine Problem and prove some positive results. Namely, the Bounded Width  Diophantine Problem in $K(G)$ is decidable if and only if  Diophantine problems in $K$ and $G$ are  decidable.  

 Now we would like to state the following 
\begin{conjecture} Let $G$ be a torsion-free group. Then for any field $K$ the Diophantine problem in $K(G)$ is undeciblable.
\end{conjecture}

\section{E-interpretability}

Recall that $A\subset \mathcal{M}^n$ is called {\em e-definable} (equationally definable) or {\em Diophantine}  in $\mathcal{M}$ if there exists a finite system of equations $\Sigma(x_1,\ldots,x_n,y_1, \ldots,y_m)$ such that 
$(a_1,\ldots,a_n) \in A$ if and only the system  $ \Sigma(a_1,\ldots,a_n,y_1, \ldots,y_m)$ in variables $y_1, \ldots,y_m$ has a solution in $\mathcal{M}$. In other words,  Diophantine sets are projections of algebraic sets defined by finite systems of equations.

We note that if $D_1$ and $D_2$ are Diophantine in an integral domain $R$ then intersections and unions of Diophantine sets are again Diophantine. Indeed, for polynomials $f_1, f_2$ with coefficients in $R$ one has 
 $f_1=0\vee f_2=0$ if and only if  $f_1f_2=0$, while the conjunction of two finite systems of equations is equivalent to the union of these systems. Furthermore, if $R$ is such that the field of quotients of $R$ is not algebraically closed then conjunction of two equations is equivalent to a single equation (see \cite{denef}).

Now we define an important notion of interpretations by equations.

\begin{definition}(E-interpretation or Diophantine interpretation)
Let $\mathcal{A}$ and $\mathcal{M}$ be algebraic structures. A map $\phi:X\subset \mathcal{M}^n \to \mathcal{A}$ is called an e-interpretation of $\mathcal{A}$ in  $\mathcal{M}$ if 
\begin{itemize}
\item [1)] $\phi$ is onto;
\item [2)] $X$ is e-definable in $\mathcal{M}$;
\item [3)] The preimage of $"="$ in $\mathcal{A}$ is e-definable in $\mathcal{M}$; 
\item [4)] The preimage of the graph of every function and predicate  in $\mathcal{A}$ is e-definable in $\mathcal{M}$.
\end{itemize}
\end{definition}
For algebraic structures $\mathcal A$ and $\mathcal M$ we write $\A \to_e \M$ if $\A$ is e-interpretable in $\M$. 

\medskip
\noindent
{\bf Examples.} {\it The following are known examples of e-interpretability: 
\begin{itemize}
\item [1)] $\N$ is Diophantine in $\Z$ since 
$$  x \in \mathbb{N} \Longleftrightarrow \exists y_1, \ldots, y_4 (x = y_1^2 + \ldots +y_4^2);$$ 
\item [2)] the ring $\Q$ is e-interpretable in $\Z$ as a field of fractions. 
\item [3)] $\Z$ is e-interpretable in $\N$;
\item [4)] The structure $\langle \Z ; +, \mid \rangle$, where $\mid$ is the predicate of division, is e-iterpretable in $\Z$.
\end{itemize}
}

\begin{lemma}\label{le:ideal-interp}
The following hold:
\begin{itemize}
\item [1)] Let $R$ be a ring or a group and $Y = \{y_1, \ldots ,y_n\}$ a finite subset of $R$ then the centralizer $C_R(Y) = \{x \in R \mid xy_1 = y_1x, \ldots  xy_n = y_nx\}$ is Diophantine  in $R$;
\item [2)] If $R$ is a commutative associative unitary ring and $Y$ a finite subset of $R$ then the ideal $\langle Y \rangle$ generated in $R$ by $Y$ is Diophantine  in $R$.
 \item [3)] If $I$ is a Diophantine ideal in $R$ then the quotient ring $R/I$ is e-interpretable in $R$.
\end{itemize}
\end{lemma}
\begin{proof}
The first statement follows directly from the definition of centralizers. To see 2) put $Y = \{y_1, \ldots,y_n\}$ and observe that $x \in \langle Y \rangle$ if and only if $\exists z_1 \ldots \exists z_n (x = y-1z_1 + \ldots +y_nz_n)$, hence $ \langle Y \rangle$ is Diophantine in $R$. To prove 3) it suffices to show that the equivalence relation $a \simeq_I b  \leftrightarrow a \in b+I$ is Diophantine in $R$ since the ring operations on $R/I$ are induced from $R$. To see former observe that $a \in b +I \leftrightarrow \exists z(z \in I \wedge a = b+z)$. So the relation $\simeq_I$ is defined by as a projection of the solution set of a system of  two equations. Notice that if $R$ is an integral domain then one can take a single equation.
\end{proof}

The following result is easy, but useful.
\begin{lemma}(Transitivity of Diophantine interpretation) \label{le:transitive}
For algebraic structures  $\A, \B, \C$  if $\A \to_e\B$ and $ \B \to_e \C$ then $\A \to_e \C$.
\end{lemma}

The following result gives the main technical tool in our study of Diophantine problems.
\begin{lemma}\label{le:reduction}
  Let $\phi:X\subset \mathcal{M}^n\to \mathcal{A}$ be a Diophantine interpretation of $\mathcal{A}$ in $\mathcal{M}$. Then  there is an effective procedure that given   a finite system of equations $S = 1$ over $\mathcal{A}$ constructs an equivalent system of equations $S' = 1$ over $\mathcal{M}$, such that  $\bar a$ is a solution of $S = 1$ in $\mathcal{A}$ iff  $\phi^{-1}(\bar a)$ is a solution of $S' = 1$ in $\mathcal{M}$.
  \end{lemma}

\begin{lemma}\label{le:reductions}
 Let $\A$ and $\M$ be algebraic structures such that $\A \to_e \M$. Then the following holds:
\begin{itemize}
\item [1)] if the Diophantine problem  in $\mathcal{A}$ is undecidable then it is undecidable in $\mathcal{M}$.
\item [2)] if the Diophantine problem  in $\mathcal{M}$ is decidable then it is decidable in $\mathcal{A}$.
\end{itemize}
\end{lemma}

In this connection it is interesting to return to the question on decidability of equations in $\mathbb Q$. A natural  approach might seek a  positive answer to the question by proving that  the set ${\mathbb Z}$ is Diophantine  in the field ${\mathbb Q}$, thus showing  undecidability of Diophantine problem for ${\mathbb Q}$.  But the following
observations seem to make such an expectation unlikely.
All known examples of algebraic varieties over ${\mathbb Q}$ have the property that the real topological
closure of the Zariski closure of their rational (over ${\mathbb Q}$) points has finitely many
connected components.
In consequence Mazur asked whether this is true for all algebraic varieties \cite{mazur}.
He also stated a more general similar statement (where the real topology
is replaced  by $p$-adic topologies). These questions remain open. If the Mazur's conjecture is true then  ${\mathbb Z}$ is not Diophantine  in  ${\mathbb Q}$ because finitely many components project onto finitely many components.
 Some specialists doubt the truth of Mazur's question (mainly because the
analogue of the $p$-adic version fails in global fields of positive characteristic). But still,
most specialists expect that ${\mathbb Z}$ is not Diophantine  in  ${\mathbb Q}$.

\section{Equations in various polynomial rings}

\subsection{Equations in polynomial rings}

Recall that Chebyshev's polynomials of the first and the second kind, respectively $T_n(x)$ and $U_n(x)$, are defined recurrently as integer polynomials from $\Z[x]$ as follows:
$$
T_0 = 1, \ T_1 = x, \ T_{n+1}(x)  = 2xT_n(x)  -T_{n-1}(x),
$$
$$
U_0 = 1, \ U_1 = 2x, \ U_{n+1}(x) = 2xU_n(x) - U_{n-1}(x).
$$

The following result was shown in \cite{denef}. 
\begin{lemma} \cite{denef} \label{le:Pell}
Let $R$ be a domain of zero characteristic and $R[t]$ a polynomial ring in one variable $t$ with coefficients in $R$. Then the solution set of the Pell's equation $ X^2 -(t^2-1)Y^2 = 1$ in $R[t]$ consists precisely of the pairs 
$$P = (\pm T_n,\pm U_{n-1}),  \ \ \ n \in \mathbb{N},$$
 where $T_n, U_n \in \mathbb{Z}[t]$  are  Chebyshev's polynomials of the first and the second kind respectively.
\end{lemma}
We give a sketch of the proof of this result since we will need some notation and facts in the future. 
Consider the Pell equation in $R[t]$  (the Pell curve)  
\begin{equation}\label{Pell} X^2 -(t^2-1)Y^2 = 1.\end{equation}
 Let $u$ be an element in the algebraic closure of $K[t]$ satisfying \begin{equation}\label{subst} u^2=t^2-1.\end{equation}

Then we have \begin{equation} \label{prod} (X+uY)(X-uY)=1.\end{equation}

  We parametrize the curve (\ref{subst}) by $$t=z^2+1/z^2-1,\ u=2z/z^2-1.$$

As rational functions of $z$, $X+uY, X-uY$ have poles and  zeros only at $z=\pm 1$ as follows from (\ref{prod}). Observe that $(X+uY)(-z)=(X-uY)z$, and so if $X,Y$ is a point on the Pell curve, then 
$$X+uY=c(\frac{z-1}{z+1})^m, \ X-uY=c(\frac{z-1}{z+1})^{-m} ,$$ for some $c\in K$. Substituting these two expressions into (\ref{prod}) yields $c^2=1$. 

For $c=1$ (the case $c=-1$ is similar).

 $$X+uY=(\frac{z-1}{z+1})^m =(t+u)^m. $$ From (\ref{subst}),
$(t-u)^{-m}=(t+u)^m$.

Thus 
 solutions of Pell's equation in   $\MA_K(X)$  if char$K=0$ are precisely  pairs $(\pm X_m, \pm Y_m)$ of the following type:
 
 $$X_m+uY_m=(t+u)^m$$
 $$X_m-uY_m=(t-u)^m,$$
 where  $m\in {\mathbb N}.$ These relations define precisely the Chebyshev's polynomials $X_m=T_m, Y_m=U_{m-1}.$

\begin{theorem} \label{th:e-interp-char-0}
Let $R$ be a domain of characteristic zero and $R[t_1, \ldots,t_n]$ a polynomial ring in finitely many variables  $t_1, \ldots,t_n$ with coefficients in $R$. Then the arithmetic $\Z$ is e-interpretable in the  ring $R[t_1, \ldots,t_n]$.
\end{theorem}
\begin{proof}
Observe that the ideal $I = \langle t_2, \ldots,t_n \rangle $ generated by $t_2, \ldots,t_n$ in $R[t_1, \ldots,t_n]$  is Diophantine in  $R[t_1, \ldots,t_n]$ by Lemma \ref{le:ideal-interp}. Hence again by Lemma \ref{le:ideal-interp} the ring of polynomials $R[t_1] \simeq R[t_1, \ldots,t_n]/I$ is e-interpretable in $R[t_1, \ldots,t_n]$. So by the transitivity of e-interpretations we may assume 
that $n=1$ and consider only the ring $R[t]$.

By Lemma \ref{le:Pell} the set of polynomials $S = \{\pm Y_n(t) \mid n \in \N\}$ is Diophantine in $R[t]$, since
$$
Y \in S  \Longleftrightarrow \exists X (X^2 - (t^2-1)Y^2 = 1).
$$
From the recurrent definition of the polynomials $Y_n(t)$ it follows directly that  $Y_{n}(1) = n+1$, so  the set 
$$ 
Z = \{Y(1) \mid Y \in S \} \cup \{0\} 
$$
is precisely the set of integers $\mathbb{Z}$ in $R$.
Notice that for $f, g  \in \mathbb{Z}[t]$ one has 
$$f(1) = g(1) \Longleftrightarrow \exists h (f-g = h(t-1))$$
   So the equivalence relation $f \sim g \Longleftrightarrow  f(1) = g(1)$ is Diophantine in $R[t]$. Now one can interpret by equations on $S$ the standard arithmetic operations $+, \times$ as follows   
   
   $$
   m+n = k \Longleftrightarrow Y_m + Y_n \sim Y_k,
   $$
    $$
   m \times n = k \Longleftrightarrow Y_m \times  Y_n \sim Y_k.
   $$
   This gives e-interpretation of $\Z$ in $R[t]$.
\end{proof}

 From Theorem \ref{th:e-interp-char-0},  Lemma \ref{le:reductions},  and undecidability of the  Diophantine problem in $\Z$  one has the following result. 
\begin{cor} \cite{denef} \label{co:undec-poly}
Let $R$ be a domain of characteristic zero and $R[t_1, \ldots,t_n]$ a polynomial ring in finitely many variables  $t_1, \ldots,t_n$ with coefficients in $R$. Then   Diophantine problem in $R[t_1, \ldots,t_n]$ is undecidable. 
\end{cor}

A similar  result  holds  for  integral domains of positive characteristic as well, but in this case  instead of $\Z$ one interprets a weaker structure $\langle \Z; +, \mid_p\rangle$.
Here $x\mid_ny$, by definition, means that $y = xqn^f$ for some $q, f \in\mathbb Z.$

\begin{theorem} \label{th:e-inter-char-p} \cite{denef2} 
Let $R$ be a domain of characteristic  $p > 1$ and $R[t_1, \ldots,t_n]$ a polynomial ring in finitely many variables  $t_1, \ldots,t_n$ with coefficients in $R$.  Then  $\langle \Z; +, \mid_p\rangle$ is e-interpretable in $R[t_1, \ldots,t_n]$.
\end{theorem}

The following result is an easy  generalization of Denef's results from \cite{denef,denef2} on undecidability of Diophantine problems in polynomial rings in one variable and coefficients from an in integral domain.

\begin{theorem}  \label{co:undec-poly}
Let $R$ be an integral  domain, $T$ a non-empty finite or  countable set of variables, and    $R[T]$ a polynomial ring in variables  from $T$ and  with coefficients in $R$. Then   Diophantine problem in $R[T]$ is undecidable.
\end{theorem}
\begin{proof}
It follows from Coriollary \ref{co:undec-poly} and Theorem \ref{th:e-inter-char-p} that Dophantine problem for a polynomial ring $R[t]$ in one variable $t$ and coefficients from $R$ is undecidable. Now for an element $t \in T$  consider the subring $R[t]$ in the ring $R[T]$.  Obviously, the $R$-homomorphism $\lambda: R[T] \to R[t]$ which is induced by a map $t \to t$, $T\smallsetminus \{t\} \to 0$ is a retract of $R[T]$ onto $R[t]$. Hence any equation $P(X,A) = 0$ in variables $X = \{x_1, \ldots,x_n\}$ and coefficients from $R[t]$ has a solution in $R[t]$ if and only if it has a solution in $R[T]$. Indeed, if $x_1 \to f_1, \ldots, x_n\to f_n$ is a solution of $P(X,A) = 0$ in $R[T]$ then $x_1 \to \lambda(f_1), \ldots, x_n\to \lambda(f_n)$ is a solution of $P(X,A) = 0$ in $R[t]$. Conversely,  any solution of $P(X,A) = 0$ in $R[t]$ is also a solution of $P(X,A) = 0$ in $R[T]$. This shows that Diophantine problem in $R[t]$ effectively reduces to Diophantine problem in $R[T]$, hence the latter is undecidable.
\end{proof}

\subsection{Equations in rings of Laurent polynomials}

We will use the following result that was proved in \cite{pappas} for $char(K)=0$ and in \cite{pheidas}  for arbitrary characteristic.
\begin{theorem}\cite{pappas},\cite{pheidas} Let $R$ be an integral domain. Then the following holds:
\begin{enumerate} \item If $char(R)=0$ and $i \in R$, then  $\Z[i]=\Z +i\Z$  is $e$-interpretable in $R[t,t^{-1}]$;
\item If $char(R)=0$ and $i\not\in R$, then  $\Z$  is $e$-interpretable in $R[t,t^{-1}]$;
\item If $char(R)=0$, then the Diophantine problem for $R[t,t^{-1}]$ with coefficients in $\mathbb Z [t]$ is undecidable;
\item If $char(R)=p>1$, then $\langle \Z; +, \mid_p\rangle$ is e-interpretable in $R[t,t^{-1}]$;
\item If $char(R)=p>1$, then the Diophantine problem for $R[t,t^{-1}]$ with coefficients in $K[t]$ where $K$ is the field of elements of $R$, algebraic over  ${\mathbb F}_p$, is undecidable.\end{enumerate}
\end{theorem}

\begin{proof}    
For the case $char (R)=0$ we will follow  \cite{pappas}.

Consider the Pell equation (\ref{Pell}) in $R[t,t^{-1}].$ 
  Let $u$ be an element in the algebraic closure of $R[t,t^{-1}]$ such that $u^2=t^2-1.$

Then we have \begin{equation} \label{prod1} (X+uY)(X-uY)=1.\end{equation}

By \cite{pappas}, $X+uY$ can be written in the form (as an algebraic function of $t$)
$$g(t)/t^r+\sqrt{t^2-1}f(t)/t^s,$$
with $g(t),f(t)\in R[t]$ and $r,s\in{\mathbb N}.$  One can parametrize the curve by $$t=z^2+1/z^2-1,\ u=2z/z^2-1.$$

As rational functions of $z$, $X+uY, X-uY$ have poles and zeros only at $z=\pm 1, z=\pm i$. Observe that $(X+uY)(-z)=(X-uY)z$, and so if $X,Y$ is a point on the Pell curve, then 
$$X+uY=c(\frac{z-1}{z+1})^m (\frac{z-i}{z+i})^n, \ X-uY=c(\frac{z-1}{z+1})^{-m} (\frac{z-i}{z+i})^{-n},$$ for some $c\in R$. Substituting these two expressions into (\ref{prod}) yields $c^2=1$. Consider the case $c=1$ (the case $c=-1$ is similar).
We have $$X+uY=(\frac{z-1}{z+1})^m (\frac{z-i}{z+i})^n=(t+u)^m(\frac{1-iu}{t})^n. $$ From (\ref{subst}),

$$(t-u)^{-m}=(t+u)^m, \ \  (\frac{1-iu}{t})^{-n}=(\frac{1+iu}{t})^n.$$

Thus the solution of Pell's equation is  precisely the set of pairs $(\pm X,\pm Y)$ such that
 
 $$X+uY=(t+\epsilon u)^m(\frac{1-\delta iu}{t})^n$$
 $$X-uY=(t-\epsilon u)^m(\frac{1-i\delta u}{t})^n,$$
 where $\epsilon, \delta =\pm 1, m,n\in {\mathbb N}.$

Now let $S$ denote the ring $\mathbb Z[i][t,t^{-1}]$. 
$S[u]$ is a quadratic extension of $S$. For $\epsilon,\delta$ define two sequences
$$X^{\epsilon,\delta}_{(m,n)}+uY^{\epsilon,\delta}_{(m,n)}=(t+\epsilon u)^m(\frac{1-\delta iu}{t})^n.$$
$$X^{\epsilon,\delta}_{(m,n)}-uY^{\epsilon,\delta}_{(m,n)}=(t+\epsilon u)^{-m}(\frac{1-\delta iu}{t})^{-n}.$$

It follows that for each pair $\epsilon,\delta=\pm 1$ and for every $m,n\in\N$  the pair $X^{\epsilon,\delta}_{(m,n)}, Y^{\epsilon,\delta}_{(m,n)} $ is a solution of the Pell equation, moreover, if $i\in R$ these are all the solutions and they belong to $K[t,t^{-1}]$

 Let  $char R=0$. If $i\in R$, then the solutions of the Pell's equation (\ref{Pell}) in $R[t,t^{-1}]$
are of the form above.  If $i\not\in R$, then the solutions have the same form with $n=0$.
Indeed, it only remains to show that in case $i\not\in R$, for $n\neq 0$ the  above solution does not belong to $R[t,t^{-1}].$ This follows from \cite{pappas}, Lemma 1.

Define  $V\sim W$ if the elements $V,W\in R[t,t^{-1}]$ take on the same values at $t=1$.

The relation $V\sim W$ is Diophantine over $R[t,t^{-1}]$ because $V\sim W$ iff $\exists X\in K[t,t^{-1}]$ such that $V-W=(t-1)X.$

It follows from \cite{pappas} that in the case $char (R)=0$ we have the following:  

(a) If $i\not\in R$, then 
  $\{Y(1): \exists X\  X^2-(t^2-1)Y^2=1, X,Y\in R[t,t^{-1}]  \} = \mathbb{Z}.$

(b) If $i\in R$, then  $\{Y(1): \exists X\   X^2-(t^2-1)Y^2=1, X,Y\in R[t,t^{-1}]  \} = \mathbb{Z}[i].$

Now one can interpret by equations addition and multiplication on these sets as follows
 $$
   m+n = k \Longleftrightarrow Y_m + Y_n \sim Y_k,
   $$
    $$
   m \times n = k \Longleftrightarrow Y_m \times  Y_n \sim Y_k.
   $$

To prove Theorem \ref{th:undec-eq2} in the case of positive characteristic $p$ we use \cite{pheidas}, Theorem 1.1 (i),  that implies that for an integral domain $R$ the Diophantine problem for $R[t,t^{-1}]$ is undecidable (the proof shows that 
$\langle \Z; +, \mid_p\rangle$ is e-interpretable in $R[t,t^{-1}]$).
\end{proof}

An easy generalization of this result is the following.

\begin{cor} \label{co:infinite-Laurent}
Let $R$ be an integral domain and $T$ a non-empty finite or countable set of variables.  Then Diophantine problem in the ring of Laurent polynomials  $R[T,T^{-1}]$ with coefficients in $\mathbb Z [T]$ is undecidable.
\end{cor}
\begin{proof}
The argument is similar to the one in Theorem \ref{co:undec-poly}.
\end{proof}

\section{Equations in free associative algebras}

In this section we study decidability of equations in free associative algebras  $\MA_K(X)$. 

We start with the following easy  lemmas.

\begin{lemma} \label{le:Diop_operations}
Finite disjunctions and conjunctions of  Diophantine sets in $\MA_K(A)$ are again Diophantine.
\end{lemma}
\begin{proof}
It suffices to show that finite disjunctions and conjunctions of equations in $\MA_K(A)$ are equivalent to single equations. Clearly, since $\MA_K(A)$ has no zero divisors, then if $P_1 = 0$ and $P_2 = 0$ are equations in $\MA_K(A)$, then the disjunction  $P_1 = 0 \vee P_2 = 0$ is equivalent to a single equation $P_1\cdot P_2 = 0$ in 
$\MA_K(A)$. On the other hand, the  conjunction  $P_1 = 0 \wedge P_2 = 0$  of two equations in $\MA_K(A)$ is equivalent to an equation of the type $P_1^2 + aP_2^2 = 0$ where $a$ is a constant from $A$. Indeed, in this case the degree of $P_1^2$ is even, and that of $aP_2^2 $ is odd after any substitution of  constants from $\MA_K(A)$ into variables in $P_1$ and $P_2$, unless the both sides become zero. Hence any  solution of $P_1^2 + aP_2^2 = 0$ in $\MA_K(A)$ makes also a solution of the system $P_1 = 0 \wedge P_2 = 0$ as required.

\end{proof}

\begin{lemma} \label{le:field-Diop}
The field $K$ is a Diophantine subset of $\MA_K(X)$. 
\end{lemma}
\begin{proof}
Indeed,  for $f \in \MA_K(X) $ one has 
$$
f \in K \Longleftrightarrow \exists g (fg = 1) \vee f = 0,
$$
so $K$ is a union of two Diophantine sets hence by Lemma \ref{le:Diop_operations} it is also Diophantine. 

\end{proof}

\begin{cor}
If the Diophantine problem in $K$ is undecidable then it is undecidable in $\MA_K(X)$ as well.
\end{cor}
It follows from Lemmas \ref{le:field-Diop} and  \ref{le:reductions}.  

In this section we prove the following main result.

\begin{theorem} \label{th:undec-eq1}
Let  $\MA_K(X)$ be a free associative algebra over  a field $K$.  Then  the  Diophantine problem in  $\MA_K(X)$   is undecidable.
\end{theorem}

In fact, we base our proof of Theorem \ref{th:undec-eq1} on the following results of independent interest.

\begin{theorem} \label{th:undec-eq}
Let  $\MA_K(X)$ be a free associative algebra over a field $K$. Then 
\begin{itemize}
\item if characteristic of $K$ is zero then $\Z$ has Diophantine interpretation in $\MA_K(X)$.
\item if characteristic $p$ of $K$ is positive, then the structure $\langle \Z; +, \mid_p\rangle $ has Diophantine interpretation in $\MA_K(X)$.
\end{itemize}
\end{theorem}
\begin{proof}
Suppose $char(K) = 0$. Then the ring of polynomials $K[t]$ in one variable $t$ is e-interpretable in  $\MA_K(X)$ as the centralizer $C(P)$ of a non-invertible polynomial $P \in \MA_K(X)$. Then by Theorem \ref{th:e-interp-char-0}, $\Z$ is e-iterpretable in $K[t]$. Hence by transitivity of e-interpretability (Lemma \ref{le:transitive}) $\Z$ is e-iterpretable in  $\MA_K(X)$. 

If $char(K) =p > 1$ then the result follows in a similar way from the Denef's theorem from \cite{denef2}  that states that $\langle \Z; +, \mid_p\rangle$ is e-interpretable in $K[t]$. Here $x\mid_ny$, by definition, means that $y = xqn^f$ for some $q, f \in\mathbb Z.$
\end{proof}

  Since  Diophantine problems in $\Z$ and $\langle \Z; +, \mid_n\rangle $ are undecidable (see \cite{denef2} for the second problem), it follows from Theorem \ref{th:undec-eq} and  Lemma \ref{le:reductions} that Diophantine problem in $\MA_K(X)$ is also undecidable, hence Theorem \ref{th:undec-eq1}.

\section{Equations in group algebras of hyperbolic groups}

\subsection{General facts}

Let $G$ be a torsion-free group and $K$ a field. In this section we study Diophantine problems in group algebras $K(G)$ under some restriction on $G$.

We start with some remarks  on Diophantine sets in $K(G)$. If $K(G)$ has no zero divisors then as usual unions of Diophantine sets are Diophantine. However, whether the same holds for intersections is not clear.

For a ring $R$ denote by $R^\ast$ the set of units in $R$. Recall that a  group $G$ satisfies Kaplansky's unit conjecture  if for any field $K$  units in the group algebra $K(G)$ are only the obvious ones $\alpha\cdot g$, where $\alpha \in K \smallsetminus \{0\}$ and $g \in G$.  In our case, when $G$ is torsion-free, Kaplansky's unit conjecture implies that there are no zero divisors in $K(G)$.

\begin{lemma}
Let $G$ be a a torsion-free group satisfying Kaplansky's unit conjecture. Then for any filed $K$ the following hold:
\begin{itemize}
\item [1)] the field $K$ is a Diophantine subset of $K(G)$;
\item [2)] the group $G$ is Diophantine interpretable in $K(G)$.
\end{itemize}
\end{lemma}
\begin{proof}
Observe first that the group of units $K(G)^*$ in $K(G)$ is Diophantine in $K(G)$ since it can be defined by the  formula $ \exists y (xy = 1)$ in a variable $x$.

To show 1) observe first that $K$ is  the following union  of Diophantine sets 
$$
K  = \{ 0\}\vee \{-1\}  \vee \{x \mid \exists y (x (x+1)y = 1)\}.
$$
 Indeed, all elements $K \smallsetminus \{0, -1\}$ satisfy the condition $\exists y (x (x+1)y = 1)$ so the inclusion $\subseteq $ in the equality above holds. Conversely, if $x \in K(G)$ is such that $\exists y (x (x+1)y = 1)$ in $K(G)$ then $x$ is invertible, as well as $x+1$ since $(x+1)x y = x (x+1)y = 1$. Since $K(G)^\ast = K\cdot G$ it follows that such $x$ cannot be of the form $\alpha g$ for $\alpha \in K^\ast$ and $g \in G, g \neq 1$. Hence $x \in K$, as claimed. Because $K(G)$ does not have zero divisors it follows that uniuons of Diophantine sets are Diophantine, so $K$ is Diophantine in $K(G)$.
   Now 2) follows from 1) since $K(G)^\ast = K \cdot G = K \times G$ so $G \simeq K(G)^*/K^*$ which gives a Diophantine interpretation of $G$ in $K(G)$ (because $K(G)^\ast$ and $K$ are Diophantine in $K(G)$). 

\end{proof}

\begin{theorem}
Let $G$ be a torsion-free group satisfying Kaplansky's unit conjecture. Then Diophantine problems in $K$ and $G$ effectively reduce to Diophantine problem in $K(G)$. Hence, if Diophantine problem either  in $K$  or in $G$ is undecidable then it is undecidable in $K(G)$, and if Diophantine problem is decidable in $K(G)$ then it is decidable in $K$ and in $G$.
\end{theorem}

Hence when studying decidability of Diophantine problem in the group algebra $K(G)$ which satisfies Kaplansky's unit conjecture  one can assume that Diophantine problems in $K$ and $G$ are decidable.

However, as we show below in many cases we do not need the condition on the units of $K(G)$ or on Diophantine problems in $K$ or $G$.  This is one of several surprising developments in this study.
To do this we need the following result.
\begin{lemma} \label{pr:centralizer-hyp} Let $K$ be a field,  $G$ a   torsion-free group, and $g \in G$ such that for any non-zero integer $k$  one has $C_G(g^k) = C_G(g)$.  Then  the centralizer $C_{K(G)}(g)$  of $g$ in $K(G)$ is isomorphic to the group algebra $K(C_G(g))$ of the centralizer $C_G(g)$ of  $g$ in $G$. In particular, if $C_G(g)$ is free abelian with basis $T$, then the centralizer $C_{K(G)}(g)$ is isomorphic to the ring of the Laurent  polynomials $K[T]$ and it is  e-interpretable in $K(G)$.
\end{lemma} 

\begin{proof} Suppose $1 \neq g\in G$, and for some $u \in K(G)$ one has $ug = gu$. If $u= \Sigma_{i = 1}^m \alpha _ig_i$, where $\alpha_i \in K, g_i \in G$ and all elements $g_i$ are distinct then 
$$
\Sigma_{i = 1}^m \alpha _ig_ig = \Sigma_{i = 1}^m \alpha _i gg_i.
$$
Then there is a permutation $\sigma  \in Sym(m)$  such that $\alpha _i=\alpha _{\sigma (i)}$ and $g_ig=gg_{\sigma (i)}$, so $g^{-1}g_ig=g_{\sigma (i)}.$ If $\sigma$ fixes some $i$, then $g_i$ commutes with $g$.  Now  $\sigma$ admits a decomposition as a product of disjoint  non-trivial cycles $\sigma = \sigma_1 \cdots \sigma_n$. Consider a cycle  $\sigma_i$,
 we may assume for simplicity (upon renaming  indices) that   $\sigma_i = (12\ldots k)$. Then 
$$ g^{-1}g_1g=g_2, \ldots,  g^{-1}g_{k-1}g=g_k, g^{-1}g_kg=g_1.$$
Hence  $g^{-k}g_1g^k = g_1$, i.e., 
$[g^k,g_1]=1,$ and by our assumption $[g,g_1] =1$, so $g_1=g_2=\ldots = g_k \in C_G(g)$.
 It follows that   $u = \Sigma_{i = 1}^m \alpha _ig_i$ is a linear combination of group elements that commute with $g$, so it belongs to the group algebra $K(C_G(g))$. Conversely, any element from $K(C_G(g))$ obviously commutes with $g$, so in fact $C_{K(G)}(g) = K(C_G(g))$ as claimed. 
 
 To prove the "in particular" part, observe that  the group algebra of a free abelian group with basis $T$ is a ring of Laurent polynomials in variables $T$.  This proves the lemma.
\end{proof}

\begin{theorem} \label{main} Let $K$ be a field,  $G$ a   torsion-free group, and $g \in G$ such that 
\begin{itemize}
\item the centralizer $C_G(g)$ is a countable free abelian  group;
\item   $C_G(g^k) = C_G(g)$ for any non-zero integer $k$. 
\end{itemize}
 Then   Diophantine problem in $K(G)$ is undecidable.
\end{theorem}
\begin{proof} By the Lemma \ref{pr:centralizer-hyp},  $C_{K(G)}(g)$ is isomorphic to the ring of the Laurent  polynomials $K[T,T^{-1}]$, where $T$ is a non-empty finite or countable  set of variables. Furthermore, $C_{K(G)}(g)$ is $e$-interpretable in $K(G)$ as the centralizer of $g$.   By Corollary \ref{co:infinite-Laurent} Diophantine problem in $K[T,T^{-1}]$ is undecidable. Hence by Lemma \ref{le:reductions} Diophantine problem in $K(G)$ is also undecidable.
\end{proof}

Notice, that torsion-free
nilpotent and solvable groups satisfy the condition of the theorem. We will be more interested in groups with decidable Diophantine problem that satisfy these conditions: torsion-free hyperbolic, toral relatively hyperbolic, RAAGs.

\subsection{Group algebras with undecidable Diophantine problems}

We say that the centralizer $C_G(g)$ of an element $g$ in a group $G$ is of {\em  Laurent type} if it is  a countable free abelian  group and  $C_G(g^k) = C_G(g)$ for any non-zero integer $k$. 

A relatively hyperbolic group is called {\em toral} if it is torsion-free and  parabolic subgroups are abelian.

\begin{lemma}
Let $G$ be a torsion-free hyperbolic or a toral relatively hyperbolic group. Then every proper centralizer in $G$ is of Laurent type.
\end{lemma}
\begin{proof} If $G$ is torsion-free hyperbolic or toral relatively hyperbolic, then the centralizer of every element in $G$ is a finitely generated free abelian group, in addition, $G$ is commutation transitive, and the statement follows. \end{proof}

\begin{theorem} \label{th:undec-eq2}
Let $G$ be a torsion-free  hyperbolic or a toral relatively hyperbolic group. Then for any field $K$  the Diophantine problem for $K(G)$  is undecidable.
\end{theorem}

Recall that a subgroup $H$ of a group $G$ is called {\em isolated} if the following implication holds for any $g \in G$ and $m \in \N, m\neq 0$: $g^m \in H \rightarrow g \in H$.

\begin{lemma}
Let $G$ be a right angled Artin group. Then there are elements $g \in G$ with centralizers of  Laurent type.
\end{lemma}

\begin{proof} If $G$ is a RAAG, then for any $g,h\in G$, $[g^k,h]=1$ implies $[g,h]=1$ because $[g^k,h]=1$ implies that all the letters in a reduced form of $h$ commute with all the letters in a reduced form of $g^n$.  Take $g$ to be a word containing all the canonical generators of $G$. Then by  The Centralizer Theorem in \cite{serv}, the centralizer of $g$ in $G$ is a finitely generated free abelian group, and the result follows. 
\end{proof}

\begin{cor}\label{cor:4}
Let $G$ be a right angled Artin group. Then for any field K the Diophantine problem for K(G) is undecidable.
\end{cor}

\begin{lemma}
Let $G$ be a free product of two non-trivial groups which are not both of order 2.  Then there are elements $g \in G$ with centralizers of  Laurent type.
\end{lemma}
\begin{cor} \label{cor:5} Let $G$ be a free product of two non-trivial groups which are not both of order 2.  Then 
for any field K the Diophantine problem for K(G) is undecidable.
\end{cor} 

\section{Positive results}

In this section we discuss some particular Diophantine problems that are decidable in  free associative algebras $\MA_K(A)$ and group rings $K(G)$ over hyperbolic groups $G$. We assume here for simplicity that the filed of coefficients $K$ is computable (recursive), otherwise one needs to fix   a computable subfield $K_0 \leq K$  and consider only equations from $\MA_{K_0}(A)$ and  $K_0(G)$.

\subsection{Finding solutions of bounded length}

Let $\MA_K(A)$ be a free associative algebra with coefficients in a computable filed $K$ and a finite or countable basis $A$.  Notice that in this case the algebra $\MA_K(A)$ is also computable. One important remark is due here: if the Diophantine problem is decidable in a computable ring $R$ then there is also an algorithm to find a solution to a finite system of equations if it is known that that the system has one. Indeed, in this case one can enumerate all elements in $R$ and try all of them one by one until a solution is found. 

Recall that every non-zero element $f \in \MA_K(A)$ can be written in the form $f = \Sigma_i \alpha_i M_i$ where $M_i$ are pair-wise distinct elements of the free monoid on $A$ and $\alpha_i$ are non-zero elements from $K$. This form is unique up to a permutation of the summands, we refer to it as a normal form of $f$.

A solution $x \to f_x (x \in X)$ to an equation $P(X,A) = 0$ in $\MA_K(A)$ has degree at most $m$ if all the monomials in the normal forms of all the polynomials $f_x (x \in X)$ have degree at most $m$. 

The Bounded  Degree Diophantine Problem in $\MA_K(A)$ is decidable if there exists an algorithm which given a finite system of equations $P(X,A) = 0$ in $\MA_K(A)$  and a number   $m \in \mathbb{N}$ decides whether  or not there is a solution of degree at most $m$  of the system $P(X,A) = 0$, and if so finds one.

\begin{theorem} \label{th:bounded-degree}
The Bounded  Degree Diophantine Problem in $\MA_K(A)$ is decidable if and only if the Diophantine problem in the filed $K$ is decidable. 
\end{theorem}
\begin{proof}
 Suppose an equation $P(X,A) = 0$ and a number $m \in \N$ are given. Notice first, that for a given  system $P(X,A) = 0$ one can compute  the finite set $A_0$  of all letters from the basis $A$ that occur in the system $P(X,A) =0$.  The subalgebra $\MA_K(A_0)$ is a retract of the algebra $\MA_K(A)$, so $P(X,A) = 0$ has a solution in $\MA_K(A)$ if and only if it has a solution in $\MA_K(A_0)$.  Furthermore, $P(X,A) = 0$ has a solution of degree at most $m$ in  $\MA_K(A)$ it has one like that in  $\MA_K(A_0)$ (by applying a retract that sends all elements from $A \smallsetminus A_0$ to zero). This shows that we can 
 assume  that the basis $A$ is finite. In this case there are only finitely many monomials of degree $\leq m$, say $M_1, \ldots, M_n$. So any polynomial $f \in \MA_K(A_0)$  of degree $\leq m$ has the form $\Sigma_{i = 1}^n\alpha_iM_i$, where $\alpha_i$ are "indeterminate coefficients" from $K$. To find a solution to $P(X,A) = 0$ it suffices to plug in such polynomials into $P(X,A) =0$ and solve the corresponding system for indeterminate coefficients, which results in a system of equations in $K$. It follows that if the Diophantine problem in $K$ is decidable then the bounded degree  Diophantine problem in $\MA_K(A)$ is also decidable.  Conversely, if the 
bounded degree Diophantine problem in $\MA_K(A)$ is decidable then given a finite system of equations in $K$ one can check if the system has a solution in $K$ by  solving  this system in $\MA_K(A)$ looking for  solutions of degree 0. 

\end{proof}

Let $G$ be a group generated by a finite set $A$. For an element $g \in G$ define the $A$-length $|g|_A$ of $g$ as the length of a shortest word in the alphabet $A \cup A^{-1}$ representing the element $g$ in the generators $A$. If $f = \Sigma \alpha_i g_i$ is an element of $K(G)$, where $\alpha_i \in K \smallsetminus \{0\}, g_i \in G$, then by $|f|_A$ we denote the  least total length $\Sigma_i |g_i|_A$ among all such representations of $f$. Let  $P(X,A) = 0$ be a finite system of equations  in variables from a finite set $X$ and coefficients from $K(G)$. We say that a solution $x \to f_x (x \in X)$ of $P(X,A) = 0$ in $K(G)$ is of length at most $m$ if $\Sigma_{x \in X}|f_x|_A \leq m$.

We say that the Bounded  Length  Diophantine Problem in $K(G)$ is decidable if there exists an algorithm which given a finite system of equations $P(X,A) = 0$ in $K(G)$  and a number   $m \in \mathbb{N}$ decides whether  or not there is a solution of length at most $m$  of the system $P(X,A) = 0$, and if so finds one. 

\begin{theorem}
Let $G$ be a group generated by a finite set $A$.  The Bounded  Length  Diophantine Problem in $K(G)$ is decidable if and only if the Diophantine problem in the filed $K$, as well as the word problem in $G$, are  decidable. 
\end{theorem}
\begin{proof}
 Indeed, suppose first that the Bounded  Length  Diophantine Problem in $K(G)$ is decidable. Then viewing a finite system of equations in $K$ as a system of equations in $K(G)$ and looking for solutions of degree 0 one can decide if this system has a solution in $K$, and if so find a solution. Also, given a word $w$ in the alphabet $A \cup A^{-1}$ one can solve an equation  $w - 1= 0$  in $K(G)$. Since this equation does not have any variables it has a solution in $K(G)$ if and only if it has a solution in $G$, hence the word problem in $G$ is decidable. 

Conversely, if Diophantine problem in $K$ and the word problem in $G$ are decidable then one can solve the   Bounded  Length  Diophantine Problem in $K(G)$ using an  argument similar to the one in Theorem \ref{th:bounded-degree}. 
\end{proof}

\subsection{Finding solutions of bounded width}
\label{se:bounded-width}

The width $width(f)$ of a polynomial $f \in \MA_K(A)$ is  the number of monomials that occur in the normal form of $f$.

A solution $x \to f_x (x \in X)$ to an equation $P(X,A) = 0$ in $\MA_K(A)$ has width at most $m$ if all  polynomials $f_x (x \in X)$ have width  at most $m$. 
We say that the Bounded  Width  Diophantine Problem in $\MA_K(A)$ is decidable if there exists an algorithm which given a finite system of equations $P(X,A) = 0$ in $\MA_K(A)$  and a number   $m \in \mathbb{N}$ decides whether  or not there is a solution of width  at most $m$  of the system $P(X,A) = 0$, and if so finds one.

\begin{theorem}\label{th:solutions-bounded-width}
Let $A$ be a finite or countable set and $K$ a field with decidable  Diophantine problem. Then  the Bounded Width  Diophantine Problem in $\MA_K(A)$ is decidable.
\end{theorem}
\begin{proof}
Observe first that argying like in Theorem \ref{th:bounded-degree} one can assume that the set  $A$ is finite. Now the result follows from Makanin's result \cite{Makanin1} on the decidability of the  systems of equations in a  free  semigroup.  Indeed, if $P(x_1, \ldots,x_n,A) = 0$ is a finite system of equations in $\MA_K(X)$ then substituting  a solution $x_i \to f_i \in \MA_K(A)$ into $P(X,A) = 0$ results in $P(f_1, \ldots,f_n,A) $ which is zero in $\MA_K(A)$. Hence after collecting similar terms in $P(f_1, \ldots,f_n,A) $ all coefficients become zero.  If the solution is of bounded width $\leq m$  the collection of the terms can proceed in finitely many ways (which depends on $m$), each of which results in a system of equations $S_j = 1$ in the free monoid $A^\ast$ and the corresponding  finite system of equations  $T_j = 0$ in $K$.  Given $P(X,A) = 0$ and $m$ all such possible pairs $(S_j,T_j)$ can be found effectively.  Hence $P(X,A) = 0$ has a solution in $\MA_K(A)$ of width at most $m$ if and only if there is a pair $(S_j,T_j)$ as above where $S_i = 1$ has a solution in $A^\ast$ and $T_j = 0$ has a solution in $K$. The former can be verified by the Makanin's result and the latter by our assumption that Diophantine problem in $K$ is decidable. Since the filed $K$ is computable one can find a solution of $T_j = 0$ if it exists. 
\end{proof}

Similar result holds in $K(G)$. Recall that every non-zero element $f \in K(G)$ can be written in the form $f = \Sigma_i \alpha_i g_i$ where $g_i$ are pair-wise distinct elements of the group $G$ and $\alpha_i$ are non-zero elements from $K$. This form is unique up to a permutation of the summands, we refer to it as a normal form of $f$. The width $width(f)$ of an element $f \in K(G)$ is  the number of summands in a normal form of $f$.  
A solution $x \to f_x (x \in X)$ to an equation $P(X,A) = 0$ in $K(G)$ has width at most $m$ if all  elements  $f_x (x \in X)$ have width  at most $m$.  Similar to the case of free associative algebras one can define the Bounded Width Diophantine Problem in $K(G)$.

\begin{theorem} \label{th:bounded-width-groups}
Let $G$ be a finitely generated group and $K$ a field. The Bounded Width  Diophantine Problem in $K(G)$ is decidable if and only if  Diophantine problems in $K$ and $G$ are  decidable. 
\end{theorem}
\begin{proof} 
If  Diophantine problems in $K$ and $G$ are  decidable then one can show using an argument  similar to the one  in Theorem  \ref{th:solutions-bounded-width} that the Bounded Width  Diophantine Problem in $K(G)$ is decidable. 
Conversely, if the Bounded Width  Diophantine Problem in $K(G)$ is decidable then one can solve finite systems of equations in $G$, by viewing them as systems in  $K(G)$ and looking for solutions of the width one.
Similarly, a finite system  $T = 0$ of equations in $K$ can be viewed as a system of equations in $K(G)$, in this case a homomorphism $G \to 1$ gives rise to a $K$-linear homomorphism $\varepsilon : K(G) \to K$ which is a retract on $K$. Hence if $T = 0$ has a solution in $K(G)$ then applying $\varepsilon$ to this solution one gets a solution of $T = 0$ in $K$. Hence, $T = 0$ has a solution $K$ if and only if it has a solution in $K(G)$. So Diophantine problem in $K$ is also decidable.
\end{proof}

\end{document}